\documentclass[amstex, 11pt]{amsart}
\usepackage{graphicx}
\newtheorem{thm}{Theorem}[section]
\newtheorem{Pro}{Proposition}[section]
\newtheorem{lemma}{Lemma}[section]

\newtheorem{Mythm}{Theorem}

\theoremstyle{definition}

\theoremstyle{definition}
\newtheorem{Rem}{Remark}[section]
\numberwithin{equation}{section}

\begin{document}
\title[Moduli spaces of generalized Cantor sets]{Uniform domains and moduli spaces \\  of generalized Cantor sets}
\author{
    Hiroshige Shiga    
}
\address{Professor Emeritus at Tokyo Institute of Technology} 

\email{shiga@cc.kyoto-su.ac.jp}
\date{\today}    
\keywords{Cantor set, Quasiconformal mapping}
\subjclass[2010]{Primary 30C62; Secondary 30F45}
\thanks{
The author was partially supported by the Ministry of Education, Science, Sports
and Culture, Japan;
Grant-in-Aid for Scientific Research (C), 22K03344.}

\begin{abstract}
We consider a generalized Cantor set $E(\omega)$ for an infinite sequence $\omega=(q_n)_{n=1}^{\infty}\in (0, 1)^{\mathbb N}$, and consider the moduli space $M(\omega)$ for $\omega$ which are the set of $\omega'$ for which $E(\omega')$ is conformally equivalent to $E(\omega)$.

In this paper, we may give a necessary and sufficient condition for $D(\omega):=\mathbb C\setminus E(\omega)$ to be a uniform domain. 
As a  byproduct, we give a condition for $E(\omega)$ to belong to $M(\omega_0)$, the moduli space of the standard middle one-third Cantor set. 
We also show that the volume of the moduli space $M(\omega)$ with respect to the standard product measure on $(0, 1)^{\mathbb N}$ vanishes under a certain condition for $\omega$.

\end{abstract}
\maketitle
\section{Introduction}
Let $\omega=(q_n)_{n=1}^{\infty}=(q_1, q_2, \dots)\in (0, 1)^{\mathbb N}$ be an infinite sequence of positive numbers with $0<q_n<1$ $(n\in \mathbb N)$.
We may create a Cantor set, which is denoted by $E(\omega)$, from the sequence $\omega$.
The standard middle one-third Cantor set $E(\omega_0)$ is given by $\omega_0=(\frac{1}{3})_{n=1}^{\infty}=(\frac{1}{3}, \frac{1}{3}, \dots )$.

The construction of $E(\omega)$ is the same as that of $E(\omega_0)$.
Namely, we start with the interval $I:=[0, 1]$ and remove an open interval $J_1$ with the length $q_1$ so that $I\setminus J_1$ consists of two closed intervals $I_1^1$ and $I_1^2$ of the same length.
We put $E_1=\bigcup_{i=1}^{2}I_1^{i}=I\setminus J_1$.
We remove an open interval of length $|I_1^{i}|q_2$ from each $I_1^{i}$ so that the remainder $E_2$ consists of four closed intervals of the same length, where $|J|$ is the length of an interval $J$.
Inductively, we define the set of the $(k+1)$-th step, $E_{k+1}$, from $E_k=\bigcup_{i=1}^{2^k}I_k^{i}$ by removing an open interval of length $|I_k^{i}|q_{k+1}$ from each closed interval $I_k^{i}$ of $E_k$ so that $E_{k+1}$ consists of $2^{k+1}$ closed intervals of the same length.
A Cantor set $E(\omega)$ for $\omega$ is defined by
\begin{equation*}
	E(\omega)=\bigcap_{k=1}^{\infty}E_k.
\end{equation*}
We call $E(\omega)$ \emph{the generalized Cantor set} for $\omega$.

We say that two generalized Cantor sets $E(\omega_1), E(\omega_2)$ are \emph{quasiconformally equivalent} if there exists a quasiconformal mapping $\varphi : {\mathbb C}\to {\mathbb C}$ such that $\varphi (E(\omega_1))=E(\omega_2)$.
For each $\omega=(q_n)_{n=1}^{\infty}$, we define \emph{the moduli space} ${M}(\omega)$ of $\omega$ as the set of $\omega '\in (0, 1)^{\mathbb N}$ so that $E(\omega')$ is quasiconformally equivalent to $E(\omega)$.
In this paper, we are interested in the moduli space $M(\omega)$ of $\omega$ from the view point of quasiconformal geometry and geometric function theory.

\medskip

In the previous paper \cite{Shiga}, we consider conditions for generalized Cantor sets to be quasiconformally equivalent, and we obtain the following theorem.
\begin{thm}[\cite{Shiga}]
\label{Thm1}
	Let $\omega=(q_n)_{n=1}^{\infty}$ and $\widetilde{\omega}=(\tilde q_n)_{n=1}^{\infty}$ be sequences with $\delta$-lower bound.
	We put
	\begin{equation}
		d(\omega, \widetilde{\omega})= \sup_{n\in \mathbb N} \max  \left\{\left | \log \frac{1-\tilde q_n}{1-q_n}\right |, |\tilde q_n -q_n|\right\}.
	\end{equation}
	\begin{enumerate}
		\item If $d(\omega, \widetilde{\omega})<\infty$, then there exists an $\exp(C(\delta)d(\omega, \widetilde{\omega}))$-quasiconformal mapping $\varphi$ on $\widehat{\mathbb{C}}$ such that $\varphi (E(\omega))=E(\widetilde{\omega})$, where $C(\delta)>0$ is a  constant depending only on $\delta$;
	\item if $\lim_{n\to \infty}\log \frac{1-\tilde q_n}{1-q_n}=0$, then $E(\widetilde\omega)$ is asymptotically conformal to $E(\omega)$, that is, there exists a quasiconformal mapping $\varphi$ on $\widehat{\mathbb C}$ with $\varphi (E(\omega))=E(\widetilde\omega)$ such that for any $\varepsilon>0$, $\varphi |_{U_{\varepsilon}}$ is $(1+\varepsilon)$-quasiconformal on a neighborhood $U_{\varepsilon}$ of $E(\omega)$. 
	\end{enumerate}	
\end{thm}
Here, we say that a sequence $\omega =(q_n)_{n=1}^{\infty}$ has \emph{$\delta$-lower bound} if there exists $\delta\in (0, 1)$ such that $q_n>\delta$ for any $n\in \mathbb N$.
Therefore, it follows from Theorem \ref{Thm1} that for $\omega=(q_n)_{n=1}^{\infty}$, $E(\omega)$ and $E(\omega_0)$ are quasiconformally equivalent if $0<\inf_n q_n\leq\sup_n q_n<1$.

On the other hand, we have seen that if $\sup_n q_n=1$, then $E(\omega)$ and $E(\omega_0)$ are not quasiconformally equivalent (\cite{Shiga} Theorem III).
Thus, it is a problem what happens if $\inf_n q_n=0$.

At the first glance, the author considered that $E(\omega)$ and $E(\omega_0)$ might not be quasiconformally equivalent if $\inf_n q_n=0$.
However, we find that it is not true.
Indeed, we obtain a necessary and sufficient condition for $\omega$ so that $E(\omega)$ is quasiconformally equivalent to $E(\omega_0)$.
To see this, we analyze the domain $D(\omega):=\mathbb C\setminus E(\omega)$ and show the following theorem (see \S 2 for the terminologies, and \S\S 3-4 for the proof).

\begin{Mythm}
\label{MyThm1}
	Let $E(\omega)$ be a generalized Cantor set for $\omega=(q_n)_{n=1}^{\infty}$.
	Then, $D(\omega):= {\mathbb C}\setminus E(\omega)$ is a uniform domain if and only if $N(\omega, \delta)<\infty$ for some $\delta\in (0, 1)$. 
	In fact, if $N_0:=N(\omega, \delta)<\infty$, then there exists a constant $c=c(N_0, \delta)$ depending only on $N_0$ and $\delta$ such that $D(\omega)$ is a $c$-uniform domain.
	In fact, we may take $c(N_0, \delta)$ as
	\begin{equation*}
	c(N_0, \delta):=9\left (\frac{2}{1-\delta}\right )^{N_0+1}\delta^{-1}.
\end{equation*}
\end{Mythm}

The above theorem enables us to determine an element $\omega$ of $M(\omega_0)$ as follows, while the statement might be obtained from a theorem of MacManus \cite{MacManus}.

We denote by $\Omega_b\subset (0, 1)^{\mathbb N}$ the set of $\omega=(q_n)_{n=1}^{\infty}$ such that $\sup_{n\in \mathbb N}q_n <1$.
\begin{Mythm}
\label{MyThm2}
	An infinite sequence $\omega=(q_n)_{n=1}^{\infty}\in (0, 1)^{\mathbb N}$ belongs to the moduli space $M(\omega_0)$ of $\omega_0=(\frac{1}{3})_{n=1}^{\infty}$
	 if and only if
	\begin{enumerate}
		\item $\omega\in \Omega_b$, and
		\item $N(\omega, \delta)<\infty$ for some $\delta\in (0, 1)$. 
	\end{enumerate}
\end{Mythm}

Let $m_{\infty}$ be the standard product probability measure of $(0, 1)^{\mathbb N}$ by the Lebesgue measure of $(0, 1)$ (cf. \cite{Bog} Chap.\ 3 for infinite products of measures).
As a subset of $(0, 1)^{\mathbb N}$, we may consider the volume $m_{\infty}(M(\omega))$ of the moduli space $M(\omega)$ of $\omega\in (0, 1)^{\mathbb N}$.
Let $\Omega_b\subset (0, 1)^{\mathbb N}$ be the set of $\omega=(q_n)_{n=1}^{\infty}$ such that $\sup_{n\in \mathbb N}q_n <1$.
Then, we will show
\begin{Mythm}
\label{thm:volume}
	For $\omega\in \Omega_b$, $m_{\infty}(M(\omega))=0$.
\end{Mythm}
The proof of the theorem is given in \S 6.

As seeing Theorem \ref{thm:volume}, we may propose the following conjecture.

\medskip
\noindent
{\bf Conjecture.} For any $\omega\in (0, 1)^{\mathbb N}$, $m_{\infty}(M(\omega))=0$.

\medskip
{\bf Acknowledgement.} The author thanks Prof. E. Kinjo for her careful reading of the manuscript and valuable comments.
He also thanks the referee for the careful reading of the manuscript and valuable comments.

\section{Preliminaries}
First of all, we define the number $N(\omega, \delta)$ in Theorems \ref{MyThm1} and \ref{MyThm2}.

For $\omega=(q_n)_{n=1}^{\infty}$ and $\delta\in (0, 1)$, we define $\omega (\delta ; i)$ $(i\in \mathbb N)$  by
\begin{equation*}
	\omega (\delta; i)= \inf \{k\in \mathbb N \mid  q_{i+k}\geq \delta\}\footnote{If $\{k\in \mathbb N\mid q_{i_k}\geq\delta\}=\emptyset$, then $\omega(\delta; i)=+\infty$.}
\end{equation*}
and $N(\omega, \delta)$ by
\begin{equation*}
	N(\omega, \delta)=\sup_{i\in \mathbb N} \omega (\delta ; i).
\end{equation*}

For example, $N(\omega_0 , \delta )=1$ for any $\delta\in (0, \frac{1}{3}]$. 
It is easy to see that if $q_n\to q >0$ as $n\to \infty$, then $N(\omega, \delta)<\infty$ for any $\delta\in (0, q)$.
On the other hand, if $q_n\to 0$ as $n\to \infty$, then $N(\omega, \delta)=\infty$ for any $\delta >0$.
Also, for $\omega =(\frac{1}{3}, \frac{1}{2}, \frac{1}{3}, \frac{1}{4}, \dots , \frac{1}{3}, (2n)^{-1}, \frac{1}{3}, (2n+2)^{-1}, \dots )$, we have $N(\omega , \delta)=1$ for any $\delta\in (0, \frac{1}{3}]$ but  $\inf_{n}q_n=0$. 

\medskip
Next, we define uniform domains (cf. \cite{Gehring-Hag}, \cite{Vaisala}).

A domain $D\subset \mathbb C$ is $\emph{uniform}$ if there exists a constant $c\geq 1$ such that  any $a, b\in D$ can be joined by a curve $\gamma$ in $D$ so that for each $z\in \gamma$,
\begin{description}
	\item [(U1)] $|\gamma|\leq c|a-b|$;
	\item [(U2)] $\min \{|\gamma_1|, |\gamma_2|\}\leq c\ \textrm{dist}(z, \partial D)$,
\end{description} 
where $\gamma_1, \gamma_2$ are connected components of $\gamma\setminus\{z\}$, $|\alpha|$ is the length of a curve $\alpha$, and $\textrm{dist}(\cdot, \cdot)$ stands for the Euclidian distance.

Typical uniform domains are the upper half plane $\mathbb H$ and the unit disk.
In fact, for $a, b\in \mathbb H$, the hyperbolic geodesic $\gamma$ connecting $a$ and $b$ satisfies (U1) and (U2) for $c=\frac{\pi}{2}$.
Namely, inequalities
\begin{equation}
\label{eqn:upper1}
	|\gamma|\leq \frac{\pi}{2}|a-b|
\end{equation}
and
\begin{equation}
\label{eqn:upper2}
	\min\{|\gamma_1|, |\gamma_2|\}\leq \frac{\pi}{2}\textrm{dist}(z, \partial \mathbb H)
\end{equation}
hold for any $z\in \gamma$.

Uniform domains play important roles in quasiconformal geometry and geometric function theory.
For example, quasi-disks are characterized as simply connected uniform domains.
One may consult the details in \cite{Gehring-Hag} and \cite{Vaisala}.

\medskip
Finally in this section, we define the uniform perfectness.
A closed set in $\mathbb C$ is \emph{uniformly perfect} if there exists a constant $c\in (0, 1)$ such that 
\begin{equation}
	E\cap\{z\in \mathbb C \mid cr<|z-a|<r\}\not=\emptyset
\end{equation}
for any $a\in E$ and $r\in (0, \textrm{diam}E)$, where $\textrm{diam}(A)$ is the Euclidean diameter of a set $A$ in $\mathbb C$.

It is easy to see that the standard Cantor set $E(\omega_0)$ is uniformly perfect.
The uniform perfectness is related to various topics.
For instance, it is known (cf. \cite{Sugawa}) that the limit sets of some Kleinian groups are uniformly perfect.

\section{Proof of Theorem \ref{MyThm1} (Part 1)}
In this section, we shall prove the sufficiency of the condition.
Namely, we show that if $N(\omega, \delta)<\infty$ for some $\delta>0$, then $D(\omega)$ is a $c$-uniform domain for some $c>1$ depending only on $N(\omega, \delta)$ and $\delta$.

\medskip
Suppose that $N_0 :=N(\omega, \delta)<\infty$ for some $\delta\in (0, 1)$.
Let $a, b$ be an arbitrary pair of points of $D(\omega)$.
We find a curve $\gamma$ connecting $a$ and $b$ which satisfies the conditions (U1) and (U2) of the definition of uniform domains.

\medskip
\noindent
{\bf Case 1: }$a, b\in \mathbb H$ (or in the lower half plane $\mathbb L$).

Let $\gamma$ be the hyperbolic geodesic $\gamma$ in $\mathbb H$ connecting $a$ and $b$.
From (\ref{eqn:upper1}), we have
\begin{equation*}
	|\gamma|\leq \frac{\pi}{2}|a-b|.
\end{equation*}
Moreover, since $\textrm{dist}(z, \partial \mathbb H)\leq\textrm{dist}(z, \partial D(\omega))$ for $z\in \gamma$, we have
from (\ref{eqn:upper2})
\begin{equation*}
	\min \{\gamma_1, \gamma_2\}\leq\frac{\pi}{2}\textrm{dist}(z, \partial\mathbb H)\leq \frac{\pi}{2}\textrm{dist}(z, \partial D(\omega)).
\end{equation*}
Hence, (U1) and (U2) are satisfied for $c=\frac{\pi}{2}$.

\medskip
\noindent
{\bf Case 2: } $a\in D(\omega)\setminus \mathbb R$ and $b\in D(\omega)\cap\mathbb R$ (or $b\in D(\omega)\setminus \mathbb R$ and $a\in D(\omega)\cap\mathbb R$).

We may assume that $a\in \mathbb H$.
Let $\gamma$ be a hyperbolic geodesic in $\mathbb H$ whose endpoints are $a$ and $b\in \partial\mathbb H$.
It is obvious that
\begin{equation*}
	|\gamma|\leq \frac{\pi}{2}|a-b|.
\end{equation*}
Hence, the condition (U1) holds for $c=\frac{\pi}{2}$.

To consider the condition (U2), we take $z\in \gamma$.
We may assume that $z\not= a, b$,
 and put $\gamma\setminus \{z\}=\gamma_1\cup\gamma_2$.
We label $\gamma_1$ as $\gamma_1\ni a$.
For a sufficiently small $\varepsilon >0$, we take $z_{\varepsilon}\in \gamma_2$ so that $|\gamma_{\varepsilon}|=\varepsilon$, where $\gamma_{\varepsilon}$ is the connected component of $\gamma_2\setminus\{z_{\varepsilon}\}$ containing $b$.
We may assume that $z\in \gamma -\gamma_{\varepsilon}$.

By applying (\ref{eqn:upper2}) to the geodesic $\gamma   -\gamma_{\varepsilon}$, we obtain
\begin{equation*}
	\min \{|\gamma_1|, |\gamma_2 - \gamma_{\varepsilon}|\}\leq \frac{\pi}{2}\textrm{dist}(z, \partial \mathbb H)\leq\frac{\pi}{2}\textrm{dist}(z, \partial D(\omega)).
\end{equation*}
Thus, we have
\begin{equation*}
	\min \{|\gamma_1|, |\gamma_2|\}=\lim_{\varepsilon\to 0}\min \{|\gamma_1|, |\gamma_2-\gamma_{\varepsilon}|\}\leq \frac{\pi}{2}\textrm{dist}(z, \partial D(\omega)).
\end{equation*}
Hence, the condition (U2) still holds for $c=\frac{\pi}{2}$.

\medskip
\noindent
{\bf Case 3: }$a, b\in D(\omega)\cap \mathbb R$.

Let $\gamma$ be the hyperbolic geodesic joining $a$ and $b$. Then, the same argument as in Case 2 works and we may show that the conditions (U1) and (U2) hold for $c=\frac{\pi}{2}$.

\medskip
\noindent
{\bf Case 4: }$a\in \mathbb H$ and $b\in \mathbb L$ (or $a\in \mathbb L$ and $b\in \mathbb H$).

We will show that there exists a constant $c\geq 1$ such that for any $a\in \mathbb H$ and $b\in \mathbb L$ there exists a curve $\gamma$ connecting $a$ and $b$ in $D(\omega)$ so that (U1) and (U2) hold for $c$ and $\gamma$.

We may assume that $\textrm{Im }a=\textrm{dist}(a, \mathbb R)\geq \textrm{dist}(b, \mathbb R)=|\textrm{Im }b|$.
We put $b_{\infty}=\textrm{Re }b$ and $d=\textrm{dist}(b, \mathbb R)=|b-b_{\infty}|$.

First, we suppose that $b_{\infty}\in E(\omega)$.
We may assume that $b_{\infty}\in I_1^1$.
%

\medskip
\noindent
{\bf Case 4-(i).} $b_{\infty}\in E(\omega)$ and $d>|I_1^1|=\frac{1}{2}|1-q_1|$.

Put $\tilde a=-d+d\sqrt{-1} , \tilde b=-d-d\sqrt{-1}$.
Let $\tilde \gamma$ be the union of two line segments, $\overline{b\tilde b}$ and $\overline{\tilde b\tilde a}$.
We also take the hyperbolic geodesic $\hat \gamma$ in $\mathbb H$ connecting $a$ and $\tilde a$.
Then, $\gamma:=\tilde \gamma\cup\hat \gamma$ is a curve in $D(\omega)$ connecting $a$ and $b$ (\textsc{Figure 1}).

Let $z$ be a point on $\gamma$ and $\gamma\setminus\{z\}=\gamma_1\cup\gamma_2$, where $\gamma_1$ and $\gamma_2$ are components of $\gamma\setminus\{z\}$ with $|\gamma_1|\leq |\gamma_2|$.

\begin{figure}[htbp]
	\includegraphics[width=15cm]{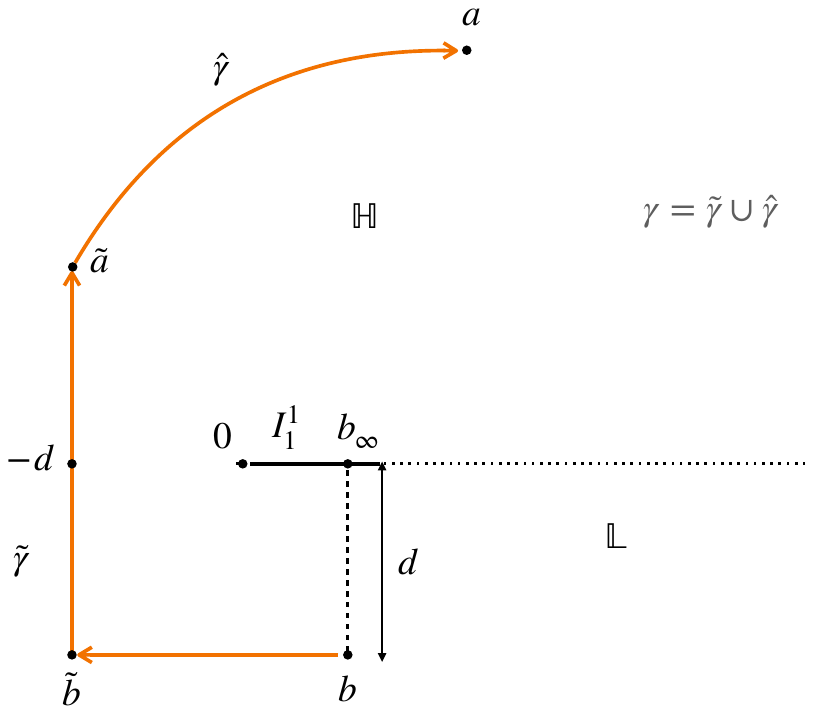}
	\caption{}
\end{figure}

If $z\in \tilde \gamma$, then we have
\begin{equation*}
	|\gamma_1|\leq |\tilde \gamma_z|\leq |\tilde \gamma|\leq 3d+|I_1^1|\leq 4d\leq 4\textrm{dist}(z, \partial D(\omega)),
\end{equation*}
where $\tilde \gamma_z$ is the component of $\tilde \gamma\setminus\{z\}$ containing $b$.

If $z\in \hat \gamma$ and $a\in\gamma_1$, then from (\ref{eqn:upper2})
\begin{equation*}
	|\gamma_1|\leq \frac{\pi}{2} \textrm{dist}(z, \mathbb R)\leq \frac{\pi}{2}\textrm{dist}(z, \partial D(\omega)).
\end{equation*}

If $z\in \hat \gamma$ and $a\not\in\gamma_1$, then $|\gamma_1|=|\hat \gamma_1|+|\tilde \gamma|$, where $\hat \gamma_1$ is the component of $\hat \gamma\setminus\{z\}$ with $\hat \gamma_1\ni \tilde a$.
We have
\begin{equation}
\label{eqn:3-1}
	|\hat \gamma_1|\leq \frac{\pi}{2}\textrm{dist}(z, \mathbb R)\leq \frac{\pi}{2}\textrm{dist}(z, \partial D(\omega))
\end{equation}
from (\ref{eqn:upper2}).
Noting that $\textrm{Im }a\geq |\textrm{Im }b|=d=\textrm{Im }\tilde a$, we see that $\textrm{dist}(z, \mathbb R)\geq d$.
Hence,
\begin{equation}
\label{eqn:3-2}
	|\tilde \gamma|\leq 4d\leq 4\textrm{dist}(z, \mathbb R)\leq 4\textrm{dist}(z, \partial D(\omega)).
\end{equation}
It follows from (\ref{eqn:3-1}) and (\ref{eqn:3-2}) that
\begin{equation*}
	|\gamma_1|=|\tilde \gamma|+|\hat \gamma| \leq \left (\frac{\pi}{2}+4\right )\textrm{dist}(z, \partial D(\omega)).
\end{equation*}

In any case, we verify that the curve $\gamma=\hat \gamma\cup\tilde \gamma$ is a curve connecting $a$ and $b$ in $D(\omega)$ which satisfies the condition (U2) for $c=\frac{\pi}{2}+4$.

We consider the condition (U1) for $\gamma$.
Since $2d\leq |a-b|$, we have 
\begin{equation*}
	|\hat \gamma|\leq \frac{\pi}{2}|a-\tilde a|\leq \frac{\pi}{2}\left (|a-b|+\tilde \gamma\right )\leq \frac{\pi}{2}\left (|a-b|+4d\right )\leq \frac{3\pi}{2}|a-b|.
\end{equation*}
Hence,
\begin{equation*}
	|\gamma|=|\hat \gamma|+|\tilde \gamma|\leq \frac{3\pi}{2}|a-b|+4d\leq \left (\frac{3\pi}{2}+2\right )|a-b|.
\end{equation*}
Thus, we see that $\gamma$ satisfies the condition (U1) for $c=\frac{3\pi}{2}+2>\frac{\pi}{2}+4$, and it satisfies the conditions (U1) and (U2) for $c=\frac{3\pi}{2}+2$.

\medskip

\noindent
{\bf Case 4-(ii)} $b_{\infty}\in E(\omega)$ and $d\leq |I_1^1|=\frac{1}{2}|1-q_1|$.

Let $K\in \mathbb N$ be the minimal number with $|I_K^1|\leq d$. Since $b_{\infty}\in E(\omega)$, there exists a decreasing sequence $\{I_n^{\epsilon_n}\}_{n=1}^{\infty}$ of closed intervals $I_n^{\epsilon_n}$ $(\epsilon_n\in \{1, \dots , 2^{n} \})$ such that $\{b_{\infty}\}=\bigcap_{n=1}^{\infty}I_n^{\epsilon_n}$.
We put $I_n:=I_n^{\epsilon_n}$.

From the definition of $N(\omega, \delta)$, we may find $N'\in \{1, 2, \dots , N(\omega, \delta)\}$ such that $q_{K+N'}\geq\delta$ but $0<q_{K+i}<\delta$ for $i=1, \dots , N'-1$.
We have
\begin{equation}
\label{eqn:3-3}
	d\geq |I_{K+N'-1}|=|I_{K-1}|\prod_{i=0}^{N'}\frac{1-q_{K+i}}{2}\geq \left (\frac{1-\delta}{2}\right )^{N'+1}d,
\end{equation}
and
\begin{equation}
\label{eqn:3-4}
	|I_{K+N'}|\leq |I_{K}|\leq d.
\end{equation}
Let $J_{K+N'}(\subset I_{K+N'-1})$ be the open interval appearing at the $(K+N')$-th step of the construction of $E(\omega)$.
Then, we have 
\begin{equation}
\label{eqn:3-5}
	\left (\frac{1-\delta}{2}\right )^{N'+1}\delta d\leq |J_{K+N'}|=q_{K+N'}|I_{K+N'-1}|\leq d.
	\end{equation}

Let $b_{\textrm{mid}}$ be the middle point of $J_{K+N'}$. 
Put $\tilde a=b_{\textrm{mid}}+d\sqrt{-1}$ and $\tilde b=b_{\textrm{mid}}-d\sqrt{-1}$. 
Let $\tilde \gamma$ be the union of two line segments, $\overline{b\tilde b}$ and $\overline{\tilde b\tilde a}$.
We also take the hyperbolic geodesic $\hat \gamma$ in $\mathbb H$ connecting $a$ and $\tilde a$.
Then, $\gamma:=\tilde \gamma\cup\hat \gamma$ is a curve in $D(\omega)$ connecting $a$ and $b$ (\textsc{Figure 2}).

\begin{figure}[htbp]
	\includegraphics[width=15cm]{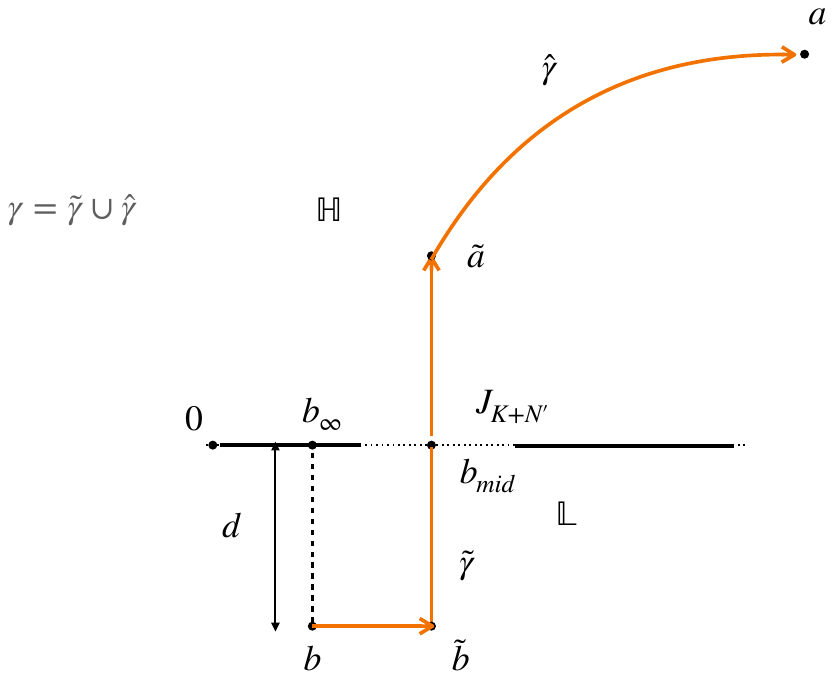}
	\caption{}
\end{figure}

Since $b_{\infty}\in I_{K+N '}$, we have from (\ref{eqn:3-3}) and (\ref{eqn:3-4})
\begin{equation*}
	|\tilde \gamma|\leq |I_{K+N'}|+\frac{1}{2}|J_{K+N'}|+2d\leq 4d.
\end{equation*}
Hence,
\begin{equation*}
	|\gamma|=|\hat \gamma|+|\tilde \gamma|\leq \frac{\pi}{2}|a-\tilde a|+4d. 
\end{equation*}
Also, we have
\begin{equation*}
	|a-\tilde a|\leq |a-b|+|\tilde \gamma|\leq |a-b|+4d \leq 3|a-b|
\end{equation*}
because $2d\leq |a-b|$.
Thus, we obtain
\begin{equation}
\label{eqn:3-0}
	|\gamma|\leq \left (\frac{3\pi}{2}+2\right )|a-b|.
\end{equation}
The curve $\gamma$ satisfies the condition (U1) for $c=\frac{3\pi}{2}+2$.

Now, we consider the condition (U2).
Let $z$ be a point on $\gamma$ and $\gamma\setminus\{z\}=\gamma_1\sqcup\gamma_2$, where $\gamma_1$ and $\gamma_2$ are components of $\gamma\setminus\{z\}$ with $|\gamma_1|\leq |\gamma_2|$.

(i) If $z\in \hat \gamma$, then we put $\hat \gamma\setminus \{z\}=\hat \gamma_1 \sqcup \hat \gamma_2$, where $\hat \gamma_1\ni a$. 
\begin{description}
	\item [(i)-a] If $|\hat \gamma_1|\leq |\hat \gamma_2|$, then from (\ref{eqn:upper2})
\begin{equation*}
	|\gamma_1|=|\hat \gamma_1| \leq \frac{\pi}{2}\textrm{dist} (z, \partial \mathbb H )\leq \frac{\pi}{2}\textrm{dist}(z, \partial D(\omega)).
\end{equation*}
 \item [(i)-b] If $|\hat \gamma_1| > |\hat \gamma_2|$, $|\hat \gamma_1|\leq |\hat \gamma_2|$, then from (\ref{eqn:upper2})
\begin{equation*}
	|\hat \gamma_2| \leq \frac{\pi}{2}\textrm{dist} (z, \partial \mathbb H )\leq \frac{\pi}{2}\textrm{dist}(z, \partial D(\omega)).
\end{equation*}
As we have seen as above,
\begin{equation*}
	|\tilde \gamma|\leq 4d,
\end{equation*}
and
\begin{equation*}
	d\leq \textrm{dist}(z, \partial D(\omega))
\end{equation*}
because $z\in \hat \gamma$.
Hence, we obtain
\begin{equation*}
	|\gamma_1|\leq |\hat \gamma_2|+|\tilde \gamma|\leq \left (\frac{\pi}{2}+4\right )\textrm{dist}(z, \partial D(\omega)).
\end{equation*}
\end{description}

(ii) If $z\in \overline{\tilde a \tilde b}$, then from (\ref{eqn:3-5}) we have
\begin{equation}
\label{eqn:3-6}
	\textrm{dist}(z, \partial D(\omega))\geq \frac{1}{2}|J_{K+N'}|\geq \frac{1}{2}\left (\frac{1-\delta}{2}\right )^{N'+1}\delta d,
\end{equation}
and
\begin{equation*}
	|\gamma_1|\leq |\overline{b\tilde b}|+d+\textrm{Im }z\leq |\overline{b\tilde b}|+2d.
\end{equation*}
Noting that
\begin{equation*}
	|\overline{b\tilde b}|\leq |I_{K+N'}|+\frac{1}{2}|J_{K+N'}|\leq 2d,
\end{equation*}
we obtain
\begin{equation}
\label{eqn:3-7}
	|\gamma_1|\leq 4d.
\end{equation}
Thus, we have from (\ref{eqn:3-6}) and (\ref{eqn:3-7})
\begin{eqnarray}
	|\gamma_1|&\leq &8\left (\frac{2}{1-\delta}\right )^{N'+1}\delta^{-1}\textrm{dist}(z, \partial D(\omega)) \nonumber \\
	\label{eqn:3-8}
	&\leq &8\left (\frac{2}{1-\delta}\right )^{N_0+1}\delta^{-1}\textrm{dist}(z, \partial D(\omega)).
\end{eqnarray}

(iii) If $z\in \overline{b\tilde b}$, then 
\begin{equation*}
	\textrm{dist}(z, \partial D(\omega))\geq d,
\end{equation*}
and
\begin{equation*}
	|\gamma_1|=|\overline{bz}|\leq |\overline{b\tilde b}|\leq 2d.
\end{equation*}
Hence,
\begin{equation*}
	|\gamma_1|\leq 2\textrm{dist}(z, \partial D(\omega)).
\end{equation*}

\medskip

Over all, from (\ref{eqn:3-0}) and (\ref{eqn:3-8}) we verify that the curve $\gamma$ satisfies both (U1) and (U2) for 
\begin{equation}
\label{eqn:const1}
	c=\max \left \{\frac{3\pi}{2}+2, 8\left (\frac{2}{1-\delta}\right )^{N_0+1}\delta^{-1}\right \}=8\left (\frac{2}{1-\delta}\right )^{N_0+1}\delta^{-1},
\end{equation}
because $\delta^{- 1}(1-\delta)^{-1}\geq 4$ if $\delta\in (0, 1)$.

\medskip

\noindent
{\bf Case 4-(iii).} $b_{\infty}\not\in E(\omega)$.

Let $J_b\subset I_1^1$ be the component of $\mathbb R\setminus E(\omega)$ containing $b_{\infty}$.

First, we suppose that $|J_b|\geq d:=|b-b_{\infty}|$.
Let $b_{\textrm{0}}$ be a point on $J_b$ with $|b_{\infty}-b_0|=\frac{d}{4}$ and $\textrm{dist}(b_0, \partial D(\omega))\geq \frac{d}{4}$.
Such a point exists because $|J_b|\geq d$.
Put $\tilde b=b_{\textrm{0}}-d\sqrt{-1}$ and $\tilde a=b_{\textrm{0}}+d\sqrt{-1}$.
Let $\tilde \gamma$ be the union of two line segments, $\overline{b\tilde b}$ and $\overline{\tilde b\tilde a}$.
We also take the hyperbolic geodesic $\hat \gamma$ in $\mathbb H$ connecting $a$ and $\tilde a$.
Then, $\gamma:=\tilde \gamma\cup\hat \gamma$ is a curve in $D(\omega)$ connecting $a$ and $b$.

\begin{figure}[htbp]
	\includegraphics[width=15cm]{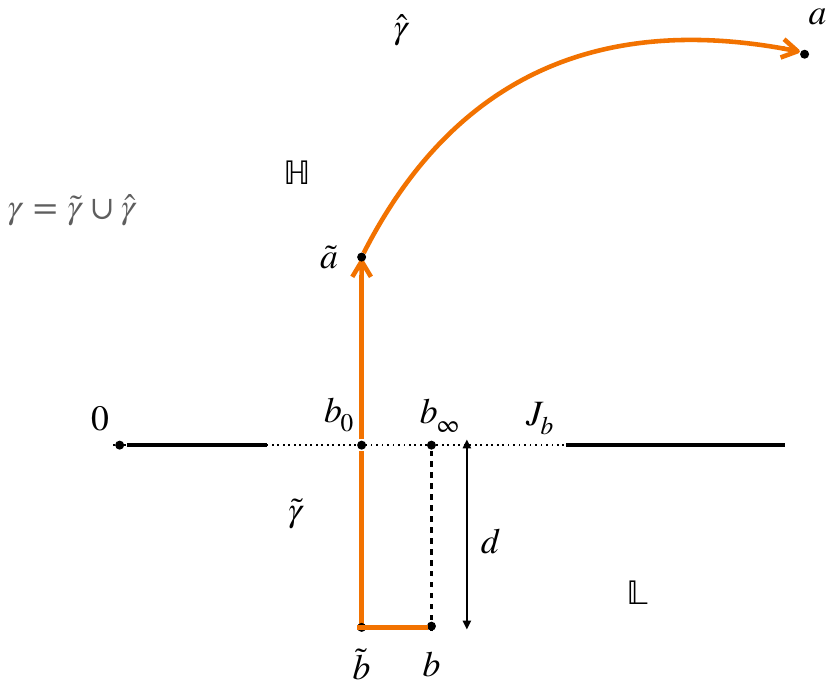}
	\caption{}
\end{figure}

We have
\begin{equation*}
	|\gamma|=|\tilde \gamma|+|\hat \gamma|\leq \frac{\pi}{2}|a-\tilde a|+2d+\frac{d}{4}<\frac{\pi}{2}|a-\tilde a|+3d,
\end{equation*}
and
\begin{equation*}
	|a-\tilde a|\leq |a-b|+|\tilde \gamma| \leq |a-b|+3d.
\end{equation*}
Hence,
\begin{equation*}
	|\gamma|\leq (2\pi +3)|a-b|,
\end{equation*}
since $d\leq |a-b|$.
Thus, we see that the curve $\gamma$ satisfies the condition (U1) for $c=2\pi +3$.

Let us consider the condition (U2).
Let $z$ be a point on $\gamma$ and $\gamma\setminus\{z\}=\gamma_1\sqcup\gamma_2$, where $\gamma_1$ and $\gamma_2$ are components of $\gamma\setminus\{z\}$ with $|\gamma_1|\leq |\gamma_2|$.

If $z\in \hat \gamma$, then we put $\hat \gamma\setminus \{z\}=\hat \gamma_1 \sqcup \hat \gamma_2$, where $\hat \gamma_1\ni a$. 
By using the same arguments as in (i)-a and (i)-b of Case 4-(ii), we obtain
\begin{equation*}
	|\gamma_1|\leq \left (\frac{\pi}{2}+3\right )\textrm{dist}(z, \partial D(\omega)).
\end{equation*}

If $z\in \hat \gamma$, then we put $\hat \gamma\setminus \{z\}=\hat \gamma_1 \sqcup \hat \gamma_2$, where $\hat\gamma_1\ni a$.
\begin{equation*}
	|\gamma_1|\leq \frac{\pi}{2}\textrm{dist}(z,  \mathbb R)\leq \frac{\pi}{2} \textrm{dist}(z, \partial D(\omega)).
\end{equation*}

If $z\in \tilde \gamma$, then
\begin{equation*}
	|\gamma_1|\leq |\tilde \gamma|\leq 3d,
\end{equation*}
and
\begin{equation*}
	\textrm{dist}(z, \partial D(\omega))\geq \textrm{dist}(b_0, \partial D(\omega))\geq \frac{d}{4}.
\end{equation*}
Thus, we obtain
\begin{equation*}
	|\gamma_1|\leq 12\textrm{dist}(z, \partial D(\omega)).
\end{equation*}

Therefore, the curve $\gamma$ satisfies both (U1) and (U2) for $c=12$.

Finally, we consider the case where $|J_b|\leq d$.
Let $b_{\infty}'\in E(\omega)$ be a point of $\partial J_b$ with
$|b_{\infty}-b_{\infty}'|\leq \frac{1}{2}|J_b|\leq \frac{d}{2}$. 

We construct a curve $\gamma'\subset D(\omega)$ connecting $a$ and $b':=b_{\infty}'-d\sqrt{-1}$ by using the same manner as in Case 4-(ii) for $a$ and $b'$.
Put $\gamma=\gamma'\cup \overline{b'b}$.

We consider the condition (U1) for $\gamma$.

It is easily seen that 
$$
|\overline{b'b}|\leq \frac{d}{2}\leq \frac{1}{4}|a-b|
$$ 
and 
$$
|a-b'|\leq |a-b|+|\overline{b'b}|\leq \frac{3}{2}|a-b|.
$$
The argument of (\ref{eqn:3-0}) in Case 4-(ii) gives us
\begin{equation*}
	|\gamma'|\leq \left (\frac{3\pi}{2}+2\right )|a-b'|.
\end{equation*}
Hence, we obtain
\begin{equation*}
	|\gamma|=|\gamma'|+|\overline{b'b}|< \left (3\pi +4\right )|a-b|.
\end{equation*}
The curve $\gamma$ satisfies (U1) for $c=3\pi +4$.

We consider the condition (U2).
Let $z$ be a point on $\gamma(=\gamma'\cup \overline{b'b})$ and $\gamma\setminus\{z\}=\gamma_1\sqcup\gamma_2$, where $\gamma_1$ and $\gamma_2$ are components of $\gamma\setminus\{z\}$ with $|\gamma_1|\leq |\gamma_2|$.

If $z\in \gamma'$, then we put $\gamma'\setminus\{z\}=\gamma_1'\sqcup \gamma_2'$ with $|\gamma_1'|\leq |\gamma_2'|$.
From the argument in Case 4-(ii), we have
\begin{equation*}
	|\gamma_1'|\leq c\textrm{dist}(z, \partial D(\omega)),
\end{equation*}
where 
$$
c=8\left (\frac{2}{1-\delta}\right )^{N_0+1}\delta^{-1}.
$$
Hence,
\begin{eqnarray*}
	|\gamma_1|&\leq &|\gamma_1'|+|b-b'|\leq c\textrm{dist}(z, \partial D(\omega))+\frac{1}{2}|J_b| \\
	&\leq & c\textrm{dist}(z, \partial D(\omega))+\frac{d}{2}.
\end{eqnarray*}
From the argument of (\ref{eqn:3-6}), we obtain
\begin{equation*}
	d\leq 2\left (\frac{2}{1-\delta}\right )^{N_0+1}\delta^{-1}\textrm{dist}(z, \partial D(\omega)).
\end{equation*}
Thus,
\begin{equation*}
	|\gamma_1|\leq \left (c+\left (\frac{2}{1-\delta}\right )^{N_0+1}\delta^{-1}\right )\textrm{dist}(z, \partial D(\omega)).
\end{equation*}

If $z\in \overline{b'b}$, then
\begin{equation*}
	|\gamma_1|\leq |z-b|\leq |b_{\infty}'-b_{\infty}|\leq\frac{d}{2}\leq\left (\frac{2}{1-\delta}\right )^{N_0+1}\delta^{-1}\textrm{dist}(z, \partial D(\omega)).
\end{equation*}
\medskip

From all arguments above, we conclude that any points $a, b$ in $D(\omega)$ are connected by a curve in $D(\omega)$ satisfying the condition (U1) and (U2) for
\begin{equation*}
	c(N_0, \delta):=c+\left ( \frac{2}{1-\delta}\right )^{N_0+1}\delta^{-1}=9\left ( \frac{2}{1-\delta}\right )^{N_0+1}\delta^{-1}.
\end{equation*}
We verify that $D(\omega)$ is a $c(N_0, \delta)$-uniform domain.
\qed

\medskip

\section{Proof of Theorem \ref{MyThm1} (Part 2)}
In this section, we show that if $N(\omega, \delta)=\infty$ for any $\delta >0$, then $D(\omega)$ is not a uniform domain.

Suppose that there exists $c\geq 1$ such that (U1) and (U2) are satisfied for any $a, b\in D(\omega)$ and for some $\gamma$.

Take $N\in \mathbb N$ so that $c<N/2$.
Since
\begin{equation*}
	\frac{x-1}{|\log (1-N^{-x})|}-\frac{x-1}{\log 2}\to \infty \quad (x\to \infty),
\end{equation*}
there exist $L>0$ and $K\in \mathbb N$ such that
\begin{equation}
\label{eqn:K}
\frac{L-1}{\log 2}\log N-3
	<K<\frac{L-1}{|\log (1-N^{-L})|}\log N-1.
\end{equation}

For those $K, L$, we take $M\in \mathbb N$ so that $q_{M+i}<N^{-L}$ for $i=0, 1, 2, \dots , K$.
The condition $N(\omega, N^{-L})=\infty$ guarantees the existence of $M$.
We may assume that $K, L, M, N$ are sufficiently large.

Now, we consider a closed interval $I_{M-1}$ in the $(M-1)$-th step of the construction on $E(\omega)$.
We remove the open interval $J_M$ with length $q_M|I_{M-1}|$ from $I_{M-1}$ and we have two closed intervals $I_M^i$ $(i=1, 2)$ with the same length.
Hence,
\begin{equation}
	|I_{M}^1|=|I_M^2|=\frac{1}{2}(1-q_M)|I_{M-1}|\geq\frac{1}{2}(1-N^{-L})|I_{M-1}|.
\end{equation}
and
\begin{equation}
	|J_M|=q_M|I_{M-1}|<N^{-L}|I_{M-1}|.
\end{equation}

Let $x_0\in \mathbb R$ be the middle point of $J_M$.
We define $a\in \mathbb H$ and $b\in \mathbb L$ by
\begin{equation*}
	a=x_0+N^{-L+2}|I_{M-1}|\sqrt{-1}, \quad b=x_0-N^{-L+2}|I_{M-1}|\sqrt{-1}.
\end{equation*}

From our assumption, there exists a curve $\gamma$ in $D(\omega)$ connecting $a$ and $b$ such that it satisfies (U1) and (U2) for $c$ $(<N)$.
We look for $z\in \gamma\cap\mathbb R$ and consider $\gamma\setminus\{z\}=\gamma_1\cup\gamma_2$ with $\gamma_1\ni a$.

First of all, if $z\not\in I_{M-1}$, then 
\begin{equation*}
	|\gamma|>|I_{M-1}|\geq N (2N^{-L+2})|I_{M-1}|>c|a-b|.
\end{equation*}
The condition (U1) does not hold for $c$.
Hence, $z$ must be in $I_{M-1}=I_{M}^1\cup J_{M}\cup I_{M}^2$.

If $z\in J_M$, then
\begin{equation*}
	c\ \textrm{dist}(z, \partial D(\omega))\leq\frac{N}{2}|J_M|\leq\frac{1}{2}N^{-L+1}|I_{M-1}|.
\end{equation*}
On the other hand, since $|\gamma_1|, |\gamma_2| \geq |a-x_0|=N^{-L+2}|I_{M-1}|$,
\begin{equation*}
	\min (|\gamma_1|, |\gamma_2|)>c\ \textrm{dist}(z, \partial D(\omega)).
\end{equation*}
The condition (U2) does not hold. Hence, $z\not\in J_M$ and $z\in I_M^1\cup I_M^2$.

We may assume that $z\in I_M^1$.
Since $z\in D(\omega)$, $z$ is contained in some $J_{M+m}\subset I_M^1$ $(m\geq 1)$, where $J_n\subset I$ is an open interval obtained in the $n$-th step of the construction of $E(\omega)$.

We have
\begin{equation}
\label{eqn:Fundametal-K1}
	|I_{M+i}^k|=\left ( \frac{1}{2}\right )^{i+1}\prod_{j=0}^{i}(1-q_{M+j})|I_{M-1}|\quad (k=1, 2),
\end{equation}
where $I_{M+i}^k$ $(k=1, 2)$ are closed intervals adjoining $J_{M+i}$.
and
\begin{equation}
\label{eqn:Fundamental-K2}
	|J_{M+i}|=q_{M+i}|I_{M+i-1}^k|.
\end{equation}
If $m\leq K$, then $q_{M+i}<N^{-L}$ for $i=0, 1, \dots , m$. 
From (\ref{eqn:Fundametal-K1}) and (\ref{eqn:Fundamental-K2}), we have
\begin{equation*}
	 \left (\frac{1-N^{-L}}{2}\right )^{m+1}|I_{M-1}|\leq |I_{M+m}^k|\leq \left (\frac{1}{2}\right )^{m+1}|I_{M-1}|
\end{equation*}
and
\begin{equation*}
	|J_{M+m}|\leq q_{M+m}|I_{M+m-1}|\leq \left (\frac{1}{2}\right )^{m}N^{-L}|I_{M-1}|.
\end{equation*}
Since $z\in J_{M+m}$, $|\gamma_1|\geq |a-z|$, $|\gamma_2|\geq |b-z|$. Hence, we have
\begin{equation}
\label{eqn:gammalength}
	|\gamma_1|, |\gamma_2|\geq |I_{M+m}^k|\geq \left (\frac{1-N^{-L}}{2}\right )^{m+1}|I_{M-1}|.
\end{equation}
On the other hand, 
\begin{equation}
\label{eqn:distance}
	\textrm{dist}(z, \partial D(\omega))\leq |J_{M+m}|\leq \left (\frac{1}{2}\right )^{m}N^{-L}|I_{M-1}|.
\end{equation}
However, from (\ref{eqn:K}), we have
\begin{equation*}
	m\leq K\leq \frac{L-1}{|\log (1-N^{-L})|}\log N-1.
\end{equation*}
This implies
\begin{equation}
\label{eqn:Fundamental-K3}
	N^{-L+1}\leq (1-N^{-L})^{K+1}\leq (1-N^{-L})^{m+1}.
\end{equation}
From (\ref{eqn:gammalength})--(\ref{eqn:Fundamental-K3}), we have
\begin{equation*}
	c\textrm{ dist}(z, \partial D(\omega))< \frac{N}{2}\textrm{dist}(z, \partial D(\omega))\leq |\gamma_1|, |\gamma_2|.
\end{equation*}
Therefore, $\gamma$ does not satisfy the condition (U2) if $\gamma\cap J_{M+m}\not=\emptyset$ for $m\in \{1, 2, \dots , K\}$.

Finally, we suppose that $z\in J_{M+m}$ for some $m>K$.
Since $J_{M+m}\subset I_{M+K}^k$ ($k=1$ or $2$), we have
\begin{equation*}
	|J_{M+m}|\leq \frac{1}{2}|I_{M+K}|\leq \left (\frac{1}{2}\right )^{K+1}|I_{M-1}|.
\end{equation*}
Hence,
\begin{equation}
\label{eqn:distance2}
	\textrm{dist}(z, \partial D(\omega))\leq\frac{1}{2}|J_{M+m}|\leq \left (\frac{1}{2}\right )^{K+2}|I_{M-1}|.
\end{equation}
On the other hand,
\begin{equation}
\label{gammalength2}
	|\gamma_1|, |\gamma_2|\geq |a-b|=2N^{-L+2}|I_{M-1}|.
\end{equation}
However, from (\ref{eqn:K}), we have
\begin{equation*}
	\frac{L-1}{\log 2}\log N-3
	<K.
\end{equation*}
This implies
\begin{equation}
\label{eqn:FundamentalK-4}
	N\left (\frac{1}{2}\right )^{K+2}<2N^{-L+2}.
\end{equation}
From (\ref{eqn:distance2})--(\ref{eqn:FundamentalK-4}), we have
\begin{equation*}
	c\ \textrm{dist}(z, \partial D(\omega))<N\textrm{dist}(z, \partial D(\omega))\leq |\gamma_1|, |\gamma_2|.
\end{equation*}
Hence, we verify that if $\gamma\cap J_{M+m}\not=\emptyset$ for $m>K$, then $\gamma$ does not satisfy the condition (U2) for $c<N$.

Thus, there is no curve $\gamma$ in $D(\omega)$ for $a, b$ satisfying the conditions (U1) and (U2), and we conclude that $D(\omega)$ is not a uniform domain.

\section{Proof of Theorem \ref{MyThm2}}
Theorem \ref{MyThm2} is an immediate consequence of the following proposition (cf. \cite{Taylor} Corollary 2.1, \cite{MNP} Theorem D).

\begin{Pro}
\label{Pro:key}
	Let $D$ be a domain in $\mathbb C$ whose complement is a Cantor set.
	The domain $D$ is quasiconformally equivalent to $D(\omega_0)$ if and only if $\mathbb C\setminus D$ is uniformly perfect and $D$ is a uniform domain.
\end{Pro}

First, we see that $E(\omega)$ is uniformly perfect if $\omega\in\Omega_b$.

If $E(\omega)$ is not uniformly perfect for $\omega\in\Omega_b$, then there exist $z_n\in E(\omega)$ and $r_n\in (0, \textrm{diam}E(\omega)]$ $(n=1, 2, \dots )$ such that
\begin{equation*}
	A_n:=\{n^{-1}r_n<|z-z_n|<r_n\}\subset D(\omega).
\end{equation*}
As $\textrm{mod}(A_n)\to\infty$ $(n\to\infty)$, $\ell_{A_n}(\{|z-z_n|=\sqrt{n^{-1}}r_n\})\to 0$, where $\ell_{A_n}(c)$ is the  length of a curve $c\subset A_n$ with respect to the hyperbolic metric of $A_n$.
From Schwartz lemma, we have $\ell_{a_n}(c)\geq \ell_{D(\omega)}(c)$ for any curve $c\subset A_n\subset D(\omega)$. 
Hence, we see that $\textrm{inj}(\omega)=0$, where $\textrm{inj}(\omega)$ is the injectivity radius of $D(\omega)$ defined by (\ref{def:inj rad}) of \S 6.
From Lemma \ref{lemma:inj+=bounded}, which is also in \S 6, we verify that $\omega\not\in \Omega_b$ and we have a contradiction.

If $N(\omega, \delta)<\infty$ for some $\delta >0$, then $D(\omega)$ is uniform (Theorem \ref{MyThm1}).
Therefore, from Proposition \ref{Pro:key}, we conclude that $D(\omega)$ is quasiconformally equivalent to $D(\omega_0)$ if $\omega\in \Omega_b$ and $N(\omega, \delta)<\infty$ for some $\delta\in (0, 1)$.

Conversely, if $\omega\not\in \Omega_b$, then it is already shown that $D(\omega)$ is not quasiconformally equivalent to $D(\omega_0)$ (\cite{Shiga} Theorem III).
From Theorem \ref{MyThm1}, $D(\omega)$ is not uniform if $N(\omega, \delta)=\infty$ for any $\delta>0$.
Hence, $D(\omega)$ is not quasiconformally equivalent to $D(\omega_0)$ from Proposition \ref{Pro:key}.

Thus, we complete the proof of the theorem.

\section{The volume of the moduli space (Proof of Theorem \ref{thm:volume})}
Let $m_{\infty}$ be the standard product probability measure of $(0, 1)^{\mathbb N}$.
As a subset of $(0, 1)^{\mathbb N}$, we may consider the volume $m_{\infty}(M(\omega))$ of the moduli space $M(\omega)$ of $\omega\in (0, 1)^{\mathbb N}$.

It follows from Theorem \ref{MyThm2} that $M(\omega_0)\subset \Omega_b$.
Moreover, we may show the following.
\begin{Pro}
If $\omega\in \Omega_b$, then $M(\omega)\subset \Omega_b$.
\end{Pro}
\begin{proof}
	For $\omega\in (0, 1)^{\mathbb N}$, we denote by $\textrm{inj}(\omega)$ the injectivity radius of $D(\omega)$, that is,
	\begin{equation}
	\label{def:inj rad}
		\textrm{inj}(\omega)=\frac{1}{2}\inf_{c\in \mathcal S}\ell_{D(\omega)} ([c]),
	\end{equation}
	where $\mathcal S$ is the set of non-trivial free homotopy classes of non-peripheral simple closed curves $c$ in $D(\omega)$ and $\ell_{D(\omega)} ([c])$ is the hyperbolic length in $D(\omega)$ of geodesic $[c]$ belonging to $c$.
	Since $E(\omega)$ has no isolated points, \lq\lq non-peripheral \rq\rq means that $\mathcal S$ does not contain a homotopy class of simple closed curves which do not shrink to $\infty$ on the Riemann sphere.
	
	First, we show
	\begin{lemma}
	\label{lemma:inj+=bounded}
		$\textrm{inj}(\omega)>0$ if and only if $\omega\in \Omega_b$.
	\end{lemma}
	\begin{proof}
	We have already seen that if $\omega\not\in \Omega_b$, then $\textrm{inj}(\omega)=0$ (see \cite{Shiga} \S 5).
	It suffices to show that if $\omega\in \Omega_b$, then $\textrm{inj}(\omega)>0$.
	
	Suppose that $\omega\in \Omega_b$ and $\textrm{inj}(\omega)=0$.
	Then, there exists a sequence $\{c_n\}_{n=1}^{\infty}$ of free homotopy classes of non-peripheral simple closed curves in $D(\omega)$ such that $\lim_{n\to\infty}\ell_{D(\omega)}([c_n])=0$.
	
	Since $\omega=(q_n)_{n=1}^{\infty} \in\Omega_b$, we may take $\delta\in (0, 1)$ so that $q_n\leq\delta$ $(n=1, 2, \dots )$.
	Let $D(\omega)=\cup_{n\in \mathbb Z}P_n$ be the {\it   natural} pants decomposition according to the construction of $D(\omega)$, namely, the boundary curves of pairs of pants of the pants decomposition are obtained from simple closed curves surrounding $I_k^i$ $(k\in \mathbb N; i=1, 2, \dots , 2^n)$ (see \cite{Shiga} \S 3).
	We may assume that $\partial P_n$ consists of closed geodesics, say $\gamma_n^1, \gamma_n^2$ and $\gamma_n^3$ for any $n\in \mathbb Z$.
	Then, we show the following:
	
	\medskip
	
	\noindent
	{\bf Claim.} There exists a constant $L=L(\delta)>0$ depending only on $\delta$ such that $\ell (\gamma_n^j)>L$ for $j\in\{1, 2, 3\}$ and for any $n\in \mathbb Z$.
	\medskip
	
	\noindent
	{\it Proof of the claim.} From the construction of the pants decomposition, there exist closed intervals $I_k, I'_k$ and open interval $J_k$ appearing in $k$-th step $E_k$ of $E(\omega)$ for some $k\in \mathbb N$ such that
	$I_{k-1}:=I_k\cup J_k\cup I'_k$ is a closed interval of $E_{k-1}$,
	\begin{equation*}
		|I_k|=|I'_k|=\frac{1-q_k}{2}|I_{k-1}|,
	\end{equation*}
	and $\gamma_n^1$ is a simple closed curve in $\mathbb C\setminus (I_{k}\cup I'_k)$ separating $I_{k}$ and $I'_{k}$.
	The domain $X_{q_k} :=\mathbb C\setminus (\partial I_k\cup\partial I'_k)$ is conformally equivalent to $Y_{q_k}:=\mathbb C\setminus \{-1, -\frac{q_k}{2}, \frac{q_k}{2}, 1\}$.
	Since $0<q_k\leq \delta$, we see that there exists a constant $L$ depending only on $\delta$ such that the hyperbolic length of the simple closed geodesic $\beta_{q_k}$ separating $\{-1, -\frac{\delta}{2}\}$ and $\{\frac{\delta}{2}, 1\}$ in $Y_{q_k}$ is not less than $L$.
	
	To show this, we consider a doubly connected domain $A_{q_k}:=\{|z|<1\}\setminus [-\frac{q_k}{2}, \frac{q_k}{2}]$ and the simple closed geodesic $\alpha_{q_k}$ in $A_{q_k}$ which separates $[-\frac{q_k}{2}, \frac{q_k}{2}]$ and $\{|z|=1\}$.
	The modulus of $A_{q_k}$ tends to zero as $q_k\searrow 0$.
	Hence, we verify that $\ell_{A_{q_k}}(\alpha_{q_k})\searrow 0$ as $q_k \searrow 0$.
	Since $A_{q_k}\subset Y_{q_k}$, it follows from the Schwarz lemma that
	\begin{equation*}
		\ell_{Y_{q_k}}([\alpha_{q_k}])\leq \ell_{Y_{q_k}}(\alpha_{q_k})\leq \ell_{A_{q_k}}(\alpha_{q_k}),
	\end{equation*}
	where $[\alpha_{q_k}]$ is the geodesic homotopic to $\alpha_{q_k}$ in $Y_{q_k}$.
	Hence, we have 
	$$
	\lim_{q_k\to 0}\ell_{Y_{q_k}}([\alpha_{q_k}])=0.
	$$
	In particular, there exists a constant $L_0>0$ depending only on $\delta$ such that 	
	\begin{equation}
	\label{eqn:6-1}
	0<\ell_{Y_{q_k}}([a_{q_k}])<L_0	
	\end{equation}
	  if $0<q_k\leq\delta$.
	
	On the other hand, the closed geodesic $\beta_{q_k}$ has to intersect $[\alpha_{q_k}]$ transversely.
	From the collar theorem (cf.\ \cite{Buser}) and (\ref{eqn:6-1}), we verify that there exists a constant $L>0$ depending only on $\delta$ such that
	\begin{equation*}
		 \ell_{Y_{q_k}}(\beta_{q_k})>L
	\end{equation*}
	if $0<q_k\leq \delta$.
	
	Since the closed curve $\gamma_n^1$ is homotopic to the simple closed geodesic separating $\partial I_k$ and $\partial I'_k$ in $X_{q_k}$ and the hyperbolic length of the closed geodesic is the same as that of $\beta_{Y_{q_k}}$, we obtain $\ell_{X_{q_k}}(\gamma_n^1)\geq \ell_{Y_{q_k}}(\beta_{q_k}) \geq L$.
	Since $X_{q_k}\supset D(\omega)$, we have
	\begin{equation*}
		\ell_{D(\omega)}(\gamma_n^1)\geq \ell_{X_{q_k}}(\gamma_n^1)\geq L
	\end{equation*}
	as desired.
	
	The same argument works for $\gamma_n^2$ and $\gamma_n^3$.
	Thus, we complete the proof of the claim. \qed
	
	\begin{Rem}
		In \cite{Kinjo} Lemma 2.2, E. Kinjo shows the claim with a more accurate estimate for $L$. 
	\end{Rem}
	
	\medskip
	
	We continue with the proof of the lemma. 
	
	Because of the claim, we see that for a sufficiently large $n$, the geodesic $[c_n]$ passes at least two pairs of pants in the pants decomposition.
	Moreover, since $[c_n]$ is a closed curve, there exists a pair of pants $P$ such that $[c_n]\cap P$ contains a geodesic arc $\alpha$ whose end points are on the same boundary curve of $P$. 
	
	Viewing that the hyperbolic length of every component of $\partial P$ is greater than $L$ and that the pair of pants $P$ consists of two congruent right hexagons, we see from the trigonometry of hyperbolic geometry (cf.\ \cite{Buser} Chapter 2) that the hyperbolic lengths of $\alpha$ is greater than a constant $C>0$ independent of $n$.
	This implies that $\ell_{D(\omega)}([c_n])>L$ and we have a contradiction. The proof of the lemma is completed.  	
	\end{proof}
	
	Now, we prove the proposition.
	For $\omega\in \Omega_b$, we take $\omega '$ in $M(\omega)$.
	Then, there exists a quasiconformal mapping $\varphi : \mathbb C\to \mathbb C$ with $\varphi (E(\omega))=E(\omega')$.
	It follows from Wolpert's lemma (\cite{Wolpert}) that
	\begin{equation*}
		\textrm{inj}(\omega')\geq \frac{1}{K(\varphi)}\textrm{inj}(\omega),
	\end{equation*}
	where $K(\varphi)$ is the maximal dilatation of $\varphi$.
	Hence, $\textrm{inj}(\omega ')>0$ and $\omega'\in \Omega_b$ from the above lemma.
	This implies that $M(\omega)\subset \Omega_b$ if $\omega\in \Omega_b$ and the proof is completed.	
\end{proof}

We see that $\Omega_b=\bigcup_{k=2}^{\infty}\Omega (k)$, where $\Omega (k)=\{\omega=(q_n)_{n=1}^{\infty}\in (0, 1)^{\mathbb N} \mid q_n\leq 1-k^{-1}\ (n=1, 2, \dots )\}$.
Since $m_{\infty}(\Omega (k))=0$ $(k=2, 3, \dots)$, we obtain $m_{\infty}(\Omega _b)=0$.
Hence, immediately we have $m_{\infty}(M(\omega))\leq m_{\infty}(\Omega_b)=0$ if $\omega\in \Omega_b$.
The proof of Theorem \ref{thm:volume} is completed.

\end{document}